\documentclass[12pt,amstex]{amsart}

\usepackage{mathptmx}
\usepackage{mathrsfs}
\usepackage{verbatim}
\usepackage{url}
\usepackage[all]{xy}
\usepackage{color}

\usepackage{tikz}
\usetikzlibrary{arrows.meta,animations}
\tikzset{>=stealth}
\usetikzlibrary{calc} 
\usetikzlibrary{scopes}

\usepackage[colorlinks=true,citecolor=blue]{hyperref}
\usepackage{stmaryrd}
\usepackage{epsfig}
\usepackage{amsmath,bm}
\usepackage{amssymb}
\usepackage{amssymb,amsthm,amsmath}
\usepackage{mathrsfs}
\usepackage{amscd}
\usepackage{graphicx}
\usepackage{pstricks}
\usepackage{theoremref}

\newtheorem{thm}{Theorem}[section]
\newtheorem{lem}[thm]{Lemma}

\newtheorem{prop}[thm]{Proposition}
\newtheorem{ques}{Question}

\theoremstyle{definition}

\newtheorem{exam}{Example}
\theoremstyle{remark}
\newtheorem{rem}[thm]{Remark}

\numberwithin{equation}{section}
\def \N {\mathbb{N}}

\def \Z {\mathbb{Z}}
\def \R {\mathbb{R}}

\def \S {\mathbb{S}}

\def \O {\mathcal{O}}

\def \A {\mathcal A}
\def \F {\mathcal F}
\def \G {\mathcal{G}}

\def \P {\mathcal P}

\def \ol {\overline}

\def \x {\textbf{x}}
\def \y {\textbf{y}}
\def \z {\textbf{z}}
\def \0 {\bf 0}

\begin{document}
	\title{Some dynamical properties related to polynomials}

	\author{Qinqi Wu}

	\address{Department of Mathematics, University of Science and Technology of China, Hefei, Anhui, 230026, P.R. China}
	
	\email{qinqiwu@mail.ustc.edu.cn}


	\thanks{ }
	
	\date{March 28, 2023}
	\begin{abstract}
Let $d\in\N$ and $p_i$ be an integral polynomial with $p_i(0)=0,1\leq i\leq d$. It is shown that if $S$ is thickly syndetic in $\Z$, then $\{(m,n)\in\Z^2:m+p_i(n),m+p_2(n),\ldots,m+p_d(n)\in S\}$ is thickly syndetic in $\Z^2$. 	
	
	Meanwhile, we construct a transitive, strong mixing and non-minimal topological dynamical system $(X,T)$, such that the set $\{x\in X:\forall\ \text{open}\ U\ni x,\exists\ n\in\Z \ \text{s.t.}\ T^{n}x\in U,T^{2n}x\in U\}$ is not dense in $X$.
	
	\end{abstract}
	
	\maketitle

\section{introduction}

In 1976, Furstenberg gave a dynamical proof of the Szemeredi's theorem via establishing the multiple ergodic recurrence theorem (\cite[Theorem 1.5]{F1}), which states that in any measure preserving system $(X,\mathcal{X}, \mu,T)$, any $d\in\N$ and $A\in\mathcal{X}$ with $\mu(A)>0$ there always exists $n\in\N$ such that $\mu(A\cap T^{-n}A\cap\cdots\cap T^{-dn}A)>0$. 
A conterpart in topological dynamics is the topological multiple recurrence theorem, which states that for any non-empty open set $A$ in a minimal topological system $(X,T)$ and $d\in\N$, there always exists $n\in\N$ such that $A\cap T^{-n}A\cap\cdots\cap T^{-dn}A\ne\varnothing$.

\medskip
An easy implication of the multiple recurrence theorem is the classical van der Waerden's theorem (\cite[Theorem 2.6]{F2}, \cite[Theorem 1.4]{FW}) on arithmetic progressions.
Van der Waerden's theorem states that each piecewise syndetic subset of $\Z$ contains arbitrarily long arithmetic progressions. That is, if $S\subset\Z$ is piecewise syndetic, then for all $d\in\N$, the set $\{(m,n)\in\Z^2:m,m+n,\ldots,m+(d-1)n\in S,n\ne 0\}$ is not empty. Furstenberg and Glasner \cite{FG} obtained the following beautiful result using the Stone-\v{C}ech compactification of $\Z$.

\medskip
\noindent{\bf Theorem}{\it (Furstenberg-Glasner).}
	Let $d\in\N$ and $S$ be piecewise syndetic in $\Z$, then $$\{(m,n)\in\Z^2:m,m+n,\ldots,m+(d-1)n\in S\}$$ is piecewise syndetic in $\Z^2$.

\medskip

In \cite{G}, a related system $N_d(X)$ under the action of the group $\G_d=\langle\tau_d,T^{(d)}\rangle$ generated by $\tau_d=T\times T^2\times\cdots\times T^d$ and $T^{(d)}=T\times\cdots\times T$ ($d$-times) was studied, where $N_d(X)$ is the orbit closure of $x^{\otimes d}=(x,\ldots,x)$ under $\G_d$. It was shown in \cite{G} that $(N_d(X),\langle\tau_d,T^{(d)}\rangle)$ is also minimal for a minimal topological dynamic system $(X,T)$. The Furstenberg-Glasner Theorem can be deduced by the minimality of $N_d(X)$. For more properties of $N_d(X)$, see \cite{GHSWY}, \cite{LQ}.

Later Beiglb\"{o}ck \cite{B} provided a combinatorial proof for the Furstenberg-Glasner's result just using van der Waerden's theorem. Bergelson and Hindman \cite{BH} extended this result to apply to many notions of largeness in arbitrary semigroups and to partition regular structures other than arithmetic progressions.

Recently, Huang, Shao and Ye \cite{HSY} confirmed one of the questions asked in \cite[Question 4.7]{BH}, which is a polynomial generalization of Furstenberg-Glasner's theorem:
Let $d\in\N$ and $p_i$ be an integral polynomial with $p_i(0)=0,1\leq i\leq d$, if $S$ is  piecewise syndetic in $\Z$, then $$\{(m,n)\in\Z^2:m+p_i(n),m+p_2(n),\ldots,m+p_d(n)\in S\}$$ is piecewise syndetic in $\Z^2$.
Futhermore, for a t.d.s. $(X,T)$ and a family of integral polynomials $\A=\{p_1,\ldots,p_d\}$ with $p_i(0)=0$, an induced system $N_{\infty}(X,\A)$ under the action of the group $\G:=\langle T^{\infty},\sigma\rangle$ generated by $T^{\infty}=\cdots\times T\times T\times \cdots$ and the shift map $\sigma$ was proposed in \cite{HSY}. 
It was shown that $(N_{\infty}(X,\A),\langle T^{\infty},\sigma\rangle)$ has dense minimal points if and only if so is $(X,T)$.

In this paper, we will study more dynamical properties of $N_{\infty}(X,\A)$, we have

\medskip
\noindent{\bf Theorem A.}
	Let $(X,T)$ be a topological dynamical system and $d\in\N$. Let $\A=\{p_1,\ldots,p_d\}$ be an family of integral polynomials with $p_i(0)=0,1\leq i\leq d$. Then $(X,T)$ has only one minimal point if and only if so is $(N_{\infty}(X,\A),\langle T^{\infty},\sigma\rangle)$.

\medskip

Moreover, we give an affirmative answer to one of another questions asked in \cite[Question 4.7]{BH}. That is, 

\medskip
\noindent{\bf Theorem B.}
Let $d\in\N$ and $p_i$ be an integral polynomial with $p_i(0)=0,1\leq i\leq d$. If $S$ is thickly syndetic in $\Z$, then 
\begin{equation*}
	\{(m,n)\in\Z^2:m+p_i(n),m+p_2(n),\ldots,m+p_d(n)\in S\} 
\end{equation*}
is thickly syndetic in $\Z^2$.

\medskip

By \cite[Corollary 1.8]{BL}, for any minimal topological dynamical system $(X,T)$, all integral polynomials $p_1,\ldots,p_d$ vanishing at zero, the set $\{x\in X:\forall \ \text{open }U\ni x,\exists\  n \ \text{s.t.}\ T^{p_1(n)}x\in U,\ldots,T^{p_d(n)}x\in U\}$ is dense in $X$. It is natural to ask whether the set is still dense when $(X,T)$ is just transitive. 
Unfortunately, it is not true, we state it as follows.

\medskip
\noindent{\bf Theorem C.}
There exists a non-minimal $(X,T)$ which is transitive and strong mixing, such that the set $\{x\in X:\forall \ \text{open }U\ni x,\exists\ n \ \text{s.t.}\ T^{n}x\in U,T^{2n}x\in U\}$ is a single point set, hence, not dense in $X$.

\medskip

The paper is organized as follows. In section 2, the basic notions used in the paper are introduced. In section 3, we study the properties of $N_{\infty}(X,T)$ and prove Theorem A. In section 4, we prove Theorem B. In section 5, we give a proof to Theorem C.

\medskip
\noindent{\bf Acknowledgement.} The author would like to thank Professors Song Shao and Xiangdong Ye for helpful discussions and remarks.

\bigskip
\section{Preliminaries}

In this section, we recall some basic notions and results we need in the following sections.
\subsection{Topological dynamical system}
By a topological dynamical system (t.d.s. for short) we mean a pair $(X,T)$, where $X$ is a compact metric space with a metric $\rho$ and $T:X\rightarrow X$ is a homeomorphism.
More generally, by a topological dynamical system we mean a triple $(X,G,\phi)$ with a compact metric space $X$, a topological group $G$ and a homeomorphism $\phi:G\rightarrow \text{Homeo}(X)$, where $\text{Homeo}(X)$ denotes the group of homeomorphisms of $X$. In this case, we say $G$ acts on $X$ continuously.
For brevity, we usually use $(X,G)$ to denote $(X,G,\phi)$ and use
$gx$ or $g(x)$ instead of $(\phi(g))(x)$ for $g\in G$ and $x\in X$.  Note that the system $(X,T)$ corresponds
to the case of $G$ being $\Z$.

\medskip
Let $(X,G)$ be a t.d.s. and $x\in X$. We denote the orbit of $x$ by {\it $\O(x,G)$}$=\{g^{n}x:g\in G\}$.
A point $x\in X$ is called a {\it fixed point} if $gx=x$ for any $g\in G$. We denote the set of all fixed points by $Fix(X,G)$.
A point $x\in X$ is called a {\it transitive point} if the orbit of $x$ is dense in $X$, i.e., $\overline{\O(x,G)}=X$. We denote the set of all transitive points by $Trans_{G}(X)$.

\medskip
A subset $A\subseteq X$ is called {\it invariant} if $GA=A$. When $Y\subseteq X$ is a closed and invariant subset of $(X,G)$, we say the system $(Y,G|_Y)$ is a {\it subsystem} of $(X,G)$. Usually we omit the subscript, and denote $(Y,G|_Y)$ by $(Y,G)$. 
A t.d.s. $(X,G)$ is called {\it minimal} if $X$ contains no proper non-empty closed invariant subsets. It is easy to verify that a t.d.s. is minimal if and only if every point $x\in X$ is a transitive point. In a t.d.s. $(X,G)$, we say that a point $x\in X$ is {\it minimal} if $(\overline{\O(x,G)},G)$ is minimal. The set of all minimal points of $(X,G)$ is defined by $Min(X,G)$. 

\medskip
Let $x\in X$ and $U\subseteq X$ let $N_{G}(x,U)=\{g\in G:gx\in U\}$. 
A point $x\in X$ is said to be  {\it recurrent} if for every neighborhood $U$ of $x$, $N_{G}(x,U)$ is infinite. Equivalently, $x\in X$ is recurrent if and only if there is a sequence $g_i\in G$ such that $g_{i}x\rightarrow x,i\rightarrow\infty$. The set of all recurrent points of $(X,G)$ is defined by $Rec(X,G)$. Obviously, we have $$Fix(X,G)\subset Min(X,G)\subset Rec(X,G).$$

\medskip
If $(X,T)$ and $(Y,T)$ are two t.d.s., their {\it product system} is the system $(X\times Y,T\times S)$. We write $(X^n,T^{(n)})$ for the $n$-fold product system $(X\times \cdots\times X,T\times\cdots\times T)$. The diagonal of $X^n$ is $\Delta_n(X)=\{(x,\ldots,x)\in X^n:x\in X\}$.
Let $U,V\subset X$ be two non-empty open sets, the {\it hitting time set} of $U$ and $V$ is denoted by
$$N_{T}(U,V)=\{n\in\Z:U\cap T^{-n}V\ne\emptyset\}.$$

\medskip
We say $(X,T)$ is {\it topological transitive} if for any non-empty sets $U,V\subset X$, the hitting time $N_{T}(U,V)$ is non-empty. When $X$ is a metric compact space, $(X,T)$ is transitive if and only if there is some point $x\in Trans_T(X)$.

\medskip
We say $(X,T)$ is {\it weakly mixing} if the product system $(X\times X,T\times T)$ is transitive; {\it strong mixing} if $N_{T}(U,V)$ is cofinite, i.e. there exists $N\in\N$ s.t. $U\cap T^{-n}V\ne\emptyset$, $\forall n\geq N$.

\medskip
\subsection{Furstenberg families}
Let $\P=\P(\Z)$ be the collection of non-empty subsets of $\Z$. A subset $\F$ of $\P$ is a {\it (Furstenberg) family}, if it is hereditary upwards, i.e. $F_1\subseteq F_2$ and $F_1\in\F$ imply $F_2\subseteq \F$. A family $\F$ is proper if it is a proper subset of $\P$, i.e. neither empty nor all of $\P$. If a proper family $\F$ is closed under finite intersections, then $\F$ is called a filter. For a family $\F$, the dual family is $$\F^*=\{F\in\P:Z\setminus F\notin \F\}=\{F\in\P:F\cap F'\ne\varnothing \ \text{for all}\ F'\in\F\}.$$

We say that a subset $S$ of $\Z$ is
\begin{enumerate}
	\item {\it syndetic} if it has a bounded gap;
	\item {\it thick} if it contains arbitrarily long runs of integers;
	\item {\it piecewise syndetic} if it is the intersection of a syndetic set with a thick set;
	\item {\it thickly syndetic} if for every $N\in\N$ the positions where interval with length $N$ runs begin form a syndetic set.
\end{enumerate}    
The collection of all syndetic (resp. thick, piecewise syndetic,  thickly syndetic) subsets is denoted by $\F_s$(resp. $\F_t$,$\F_{ps}$,$\F_{ts}$). We have $\F_{ts}$ is a filter and $\F_{ps}^*=\F_{ts},\F_{ts}^*=\F_{ps}.$

\medskip
Analogously by the definition in $\Z$,  we can define syndetic (resp. thick, piecewise syndetic, thickly syndetic) set in $\Z^2$.

\medskip

\subsection{Shift system, subshift}
Let $F$ be a finite set with at least two elements, say $F=\{0,1,\ldots,s-1\}$ with $s\in\N,s\geq 2$. Consider $F$ as a finite discrete topological space and put $\Omega:=F^{\Z}$. Endowed with the product topology, $\Omega_F$ is a compact metric space. The points of $\Omega_F$ are two-sided infinite sequences $x=(x_i)_{i\in\Z}$ with $x_i\in F$ for all $i\in\Z$. Define a mapping $\sigma:\Omega\rightarrow\Omega$ by $$(\sigma x_i):=x_{i+1},i\in\Z\quad \text{for} \ x=(x_i)_{i\in\Z}\in \Omega.$$
Clearly, $\sigma$ is a bijection of $\Omega$ onto itself with inverse $\sigma^{-1}$, given by $(\sigma x )_i=x_{i-1}$ for $i\in\Z$. So $\sigma$ is a homeomorphism. It is called the {\it shift}. The discrete flow $(\Omega,\sigma)$ will be called a {\it two-sided shift system}. 
$(X,\sigma)$ is called a {\it subshift} if $X$ is a closed subset of some full shift $\Omega$ that is invariant under the action of $\sigma$.
Denote by $\Z_+$ the set of all non-negative integers. Similarly, we may define {\it one-sided shift system} when $\Omega=F^{\Z^+}$. In such a case, $\sigma$ is a continuous surjective map.

\medskip
If $k\in\N$, then a {\it word of length $k$} is an element of the set $F^k:=F\times\cdots\times F$ ($k$-times). If $b\in F^k$, then $b=b_0\cdots b_{k-1}$ with $b_i\in F$ for $i=0,\ldots,k-1$. The $b_i$ are called the {\it entries} of $b$, the $j$-th entry is $b_{j-1}$. The set of finite words over $F$ will be denoted by $F^*:=\cup\{F^k:k\in\Z^+\}$ ($F^0$ consists of the empty word), it is a semigroup under the operation $(b,c)\mapsto bc:=b_0\cdots b_{|b|-1}c_0\cdots c_{|c|-1}$.
We write $b^k$ for $bb\cdots b$ ($k$-times, $k\geq 2$).

\medskip
A finite non-empty block $b$ is said to {\it occur in} a block $c$, equivalently $c$ is said to {\it contain} $b$ whenever there exist blocks $r$ and $l$ such that $c=lbr$. In this case we say that $b$ is {\it preceded by $l$} and {\it followed by $r$} in $c$. If $l=\varnothing$ (resp. $r=\varnothing$) then $c$ is said to {\it begin with $b$} (resp. {\it end with $b$}).

Let $(X,\sigma)$ be a subshift. If $x\in X$ and $j\in\Z$ then the block $x_j\cdots x_{j+n}$ is often denoted by $x[j;j+n]$. A non-empty block $b$ is said to {\it occur in $x\in X$ at place $j$} whenever $$b=x_j\cdots x_{j+|b|-1}=x[j;j+|b|-1].$$
If $b$ is a finite block and $j\in\Z$, then the {\it cylinder based on $b$ at place $j$} is the set of all elements in $\Omega$ in which $b$ occurs at palce $j$, that is, the set $$C_j[b]:=\{x\in X:x[j;j+|b|-1]=b\}.$$ Every cylinder in $\Omega$ is closed and open.

\bigskip

\section{The induced system $N_{\infty}(X,\A)$}

In this section, for a given finite set $\A$ of integral polynomials and t.d.s. $(X,T)$, we introduce the induced system $N_{\infty}(X,\A)$. We will study more properties about minimal and recurrent points of $N_{\infty}(X,\A)$. Theorem B is equivalent to Theorem \ref{min-N}.

\medskip
\subsection{The definition of $N_{\infty}(X,\A)$}
Let $d\in\N$ and $\A=\{p_1,p_2,\ldots,p_d\}$ be a collection of integral polynomials with $p_i(0)=0,1\leq i\leq d$. The point of $(X^d)^{\Z}$ is denoted by $$\x=(\x_n)_{n\in\Z}=((x^{(1)}_n,x_n^{(2)},\ldots,x_n^{(d)}))_{n\in\Z}.$$
Let $\vec{p}(n)=(p_1(n),p_2(n),\ldots,p_d(n))$ and let $T^{\vec{p}(n)}:X^d\rightarrow X^d$ be defined by $$T^{\vec{p}(n)}(x_1,x_2,\ldots,x_d)=(T^{p_1(n)}x_1,T^{p_2(n)}x_2,\ldots,T^{p_d(n)}x_d).$$

Define $T^{\infty}:(X^d)^{\Z}\rightarrow (X^d)^{\Z}$ such that $$T^{\infty}(\x_n)_{n\in\Z}=(T^{(d)}\x_n)_{n\in\Z}.$$
Let $\sigma:(X^d)^{\Z}\rightarrow (X^d)^{\Z}$
be the shift map, i.e., for all $(\x_n)_{n\in\Z}\in (X^d)^{\Z}$ $$(\sigma\x)_n=\x_{n+1},\forall n\in\Z.$$

Let $x^{\otimes d}=(x,x,\ldots,x)\in X^d$, 
$$\Delta_{\infty}(X)=\{x^{(\infty)}\triangleq(\ldots,x^{\otimes d},x^{\otimes d},\ldots)\in (X^d)^{\Z}:x\in X\}.$$
For each $x\in X$, put 
$$\omega_x^{\A}\triangleq(T^{\vec{p}(n)}x^{\otimes d})_{n\in\Z}=((T^{p_1(n)}x,T^{p_2(n)}x,\ldots,T^{p_d(n)}x))_{n\in\Z}.$$ and set 
$$N_{\infty}(X,\A)=\overline{\bigcup\{\O(\omega_x^{\A},\sigma):x\in X\}}\subseteq (X^d)^{\Z}.$$

It is clear that $N_{\infty}(X,\A)$ is invariant under the action of $T^{\infty}$ and $\sigma$, and $T^{\infty}\circ\sigma=\sigma\circ T^{\infty}$. Thus $(N_{\infty}(X,\A),\langle T^{\infty},\sigma\rangle)$ is a $\Z^2$-t.d.s.
If $(X,T)$ is transitive, then for each transitive point $x$ of $(X,T)$, $N_{\infty}(X,\A)=\overline{\O(\omega_x^{\A},\langle T^{\infty},\sigma\rangle)}$.

\bigskip

\subsection{Properties related to $N_{\infty}(X,\A)$.}

We begin with some lemmas.

\medskip
\begin{lem}\cite[Lemma 2.2(2)]{HSY}\label{TS}
	Let $(X,\langle T,S\rangle)$ be a minimal system with $T\circ S=S\circ T$, where $T,S:X\rightarrow X$ are homeomorphisms. Then the set of minimal points of $T$ (resp. of $S$) is dense in $X$.
\end{lem}

\medskip
\begin{lem}\label{fac}
	Let $\pi:(X,T)\rightarrow(Y,S)$ be a factor map. If $(X,T)$ is minimal, so is $(Y,S)$.
\end{lem}
\begin{proof}
	Let $(Y_1,T)$ be a non-empty closed subset of $Y$.
	Since $\pi_{i}^{-1}(Y_1)$ is non-empty, closed and invariant, we have $\pi_{i}^{-1}(Y_1)=X$. Then $Y_1=Y$. i.e., $(Y,S)$ is minimal.
\end{proof}

\medskip
\begin{thm}\label{iff1}
	If $Min(X,T)=\{x\}$, then $Min(N_{\infty}(X,\A),\langle T^{\infty},\sigma\rangle)=\{x^{(\infty)}\}$.
\end{thm}
\begin{proof}
	Let $M$ be a minimal subset of $N_{\infty}(X,\A)$, then $(M,\langle T^{\infty},\sigma\rangle)$ is a minimal $\Z^2$-action t.d.s.
	Denote $M_1$ be the set of  all minimal points of $T^{\infty}$ in $M$. By Lemma \ref{TS}, $M_1$ is dense in $M$.
	
	We suppose $(\y_n)_{n\in\Z}:=(y^{(1)}_n,\ldots,y^{(d)}_n)_{n\in\Z}\in M_1$.
	Then $(\ol{\O((\y_n),T^{\infty})},T^{\infty})$ is a minimal system, $\ol{\O((\y_n),T^{\infty})}$ is a closed $T^{\infty}$-invariant subset of $N_{\infty}(X,\A)$.
	
	Let $\pi_{n,i}$ be the factor map defined by
	$$\pi_{n,i}:\big(\ol{\O((\y_n),T^{\infty})},T^{\infty}\big)\rightarrow \big(\ol{\O}(y_n^{(i)},T),T\big), (\z_n)_{n\in\Z}\mapsto z_{n}^{(i)}.$$	
	
	By Lemma \ref{fac}, $(\ol{\O}(y_n^{(i)},T),T)$ is minimal, then we have $y^{(i)}_n=x$ for any $n\in\Z,1\leq i\leq d$. 
	Thus, we have $M=\ol{M_1}=\{x^{(\infty)}\}$. 
	i.e., $x^{(\infty)}$ is the only minimal point of $(N_{\infty}(X,\A),\langle T^{\infty},\sigma\rangle)$.
	
\end{proof}

\medskip
\begin{lem}\label{iff2}
	If $Min(N_{\infty}(X,\A),\langle T^{\infty},\sigma\rangle)=\{((x^{(1)}_n,x_n^{(2)},\ldots,x_n^{(d)}))_{n\in\Z}\}$ is a single point set, then $\big( (x^{(1)}_n,x_n^{(2)},\ldots,x_n^{(d)}) \big)_{n\in\Z}=x^{(\infty)}$ for some $x\in X$. Moreover,  $x$ is the unique minimal point of $(X,T)$.
\end{lem}
\begin{proof}
	Since $\big((x^{(1)}_n,x_n^{(2)},\ldots,x_n^{(d)})\big)_{n\in\Z}$ is the only minimal point, we have $$\sigma\big((x^{(1)}_n,x_n^{(2)},\ldots,x_n^{(d)})\big)_{n\in\Z}=\big((x^{(1)}_n,x_n^{(2)},\ldots,x_n^{(d)})\big)_{n\in\Z},$$
	$$T^{(\infty)}\big((x^{(1)}_n,x_n^{(2)},\ldots,x_n^{(d)})\big)_{n\in\Z}=\big((x^{(1)}_n,x_n^{(2)},\ldots,x_n^{(d)})\big)_{n\in\Z}.$$ 
	Then for any $m,n\in\Z$, we have  $x_m^{(i)}=x_n^{(i)}:=x_i$ for some $x_i\in Fix(X,T)$. It is clear that $$x_i^{(\infty)},x_j^{(\infty)}\in Fix(N_{\infty}(X,\A),\langle T^{\infty},\sigma\rangle)\subset Min(N_{\infty}(X,\A),\langle T^{\infty},\sigma\rangle).$$
	So $x_i=x_j:=x$ for any $1\leq i,j\leq d$.
	
	If $y\in Min(X,T)$, then $y^{(\infty)}\in Min(N_{\infty}(X,\A),\langle T^{\infty},\sigma\rangle)=\{x^{(\infty)}\}$ which implies that $x=y$. Hence, $Min(X,T)=\{x\}$.
\end{proof}

\medskip
\begin{thm}\label{min-N}
	Let $(X,T)$ be a t.d.s. and $d\in\N$. Let $\A=\{p_1,\ldots,p_d\}$ be an family of integral polynomials with $p_i(0)=0,1\leq i\leq d$. Then $(X,T)$ has only one minimal point if and only if so is $(N_{\infty}(X,\A),\langle T^{\infty},\sigma\rangle)$.
\end{thm}
\begin{proof}
	It follows from Lemmas \ref{iff1} and \ref{iff2}. 
\end{proof}

\medskip
\begin{rem}
		In Theorem \ref{min-N}, the action $T^{\infty}$ on $N_{\infty}(X,\A)$ is necessary.
		In general, $|Min(N_{\infty}(X,\A),\sigma)|$ does not always be 1 when $|Min(X,T)|=1$, see Example \ref{tildesigma}. 
\end{rem}

\medskip
\begin{exam}\label{tildesigma}
	Suppose that $S=\{n^2:n\in\Z\}$. Let $x\in\{0,1\}^{\Z}$ be defined by $$x(n)=1\Longleftrightarrow n\in S.$$
	Set $X=\ol{\{T^{n}x:n\in\Z\}}$ where $T$ is the left shift. Then $(X,T)$ is a subshift. 
	The point $y:={\bf 0}=\cdots000\cdots$ is the only minimal point of $(X,T)$, then $|Min(X,T)|=1$. 
	
	Let $\A=\{p_1(n)=n^2\}$. Since $x$ is a transitive point of $X$, we have
	$N_{\infty}(X,\A)=\overline{\O(\omega_x^{\A},\langle T^{\infty},\sigma\rangle)}$.
	We claim that $|Min(N_{\infty}(X,\A),\sigma)|\geq 2$.
	Note that	
	$$\omega_{x}^{\A}=(\ldots,T^{4}x,Tx,\underset{\bullet}{x},Tx,T^{4}x,\ldots)=(T^{n^2}x)_{n\in\Z},$$
	$$\sigma^{k}\omega_{x}^{\A}=(\ldots,T^{(k+2)^2}x,T^{(k+1)^2}x,\underset{\bullet}{T^{k^2}x},T^{(k-1)^2}x,\ldots)$$
	where ``$\underset{\bullet}{ }$'' means the $0$-th coordinate.
	
	Note $z:=1_{\{0\}}=\cdots0\underset{\bullet}{1}0\cdots$. For $k\in\N$, we have
	$$T^{k^2}x=\cdots 0^{(2k-2)}\underset{\bullet}{1}0^{(2k)}\cdots\ \text{and} \ \lim_{k\rightarrow+\infty}T^{k^2}x=z\in X.$$
	Then $$\lim_{k\rightarrow+\infty}\sigma^{k}\omega^{\A}_{x}= z^{(\infty)}\in N_{\infty}(X,\A) \ \text{and} \ \sigma(z^{(\infty)})=z^{(\infty)}.$$
	i.e., $z^{(\infty)}\in Fix(N_{\infty}(X,\A),\sigma)\subset Min(N_{\infty}(X,\A),\sigma)$.
	It is clear that	
	$$y^{(\infty)}=(\ldots,{\bf 0},{\bf 0},{\bf 0},\ldots)\in Min(N_{\infty}(X,\A),\sigma).$$
	
	Thus, $|Min(N_{\infty}(X,\A),\sigma)|\geq 2$.
\end{exam}

\medskip

\subsection{The case for linear polynomials.}
Let $\A=\{p_1,\ldots,p_d\}$ where $p_i(n)=a_{i}n,1\leq i\leq d$ and $a_1,\ldots,a_d$ are distinct non-zero integers. 
In this case we recall that for any t.d.s. $(X,T)$, $$(N_{\infty}(X,\A),\langle T^{\infty},\sigma\rangle)\cong (N_{\A}(X,T),\langle T^{(d)},\tau_{\vec{a}}\rangle)$$
where $N_{\A}(X,T)=\ol{\O}(\Delta_d(X),\tau_{\vec{a}})$ and $\tau_{\vec{a}}=T^{a_1}\times\cdots\times T^{a_d}$ with $\vec{a}=(a_1,a_2,\ldots,a_d)$.

As $\A=\{p_1,\ldots,p_d\}$, we have that for each $x\in X$
$$\omega_x^{\A}=\big((T^{a_{1}n},T^{a_{2}n},\ldots,T^{a_{d}n})\big)_{n\in\Z}=(\ldots,\tau_{\vec{a}}^{-1}x^{\otimes d},\underset{\bullet}{x^{\otimes d}},\tau_{\vec{a}}x^{\otimes d},\ldots)\in(X^d)^{\Z}.$$
Note that 
\begin{align*}	\sigma\omega_x^{\A}&=\big((T^{a_{1}(n+1)},T^{a_{2}(n+1)},\ldots,T^{a_{d}(n+1)})\big)_{n\in\Z}
\\	&=\big(\tau_{\vec{a}}(T^{a_{1}n},T^{a_{2}n},\ldots,T^{a_{d}n})\big)=\tau_{\vec{a}}^{\infty}\omega_x^{\A},
\end{align*}
where $\tau_{\vec{a}}^{\infty}=\cdots\times\tau_{\vec{a}}\times\tau_{\vec{a}}\times\cdots$, and for all $n\in\Z$
\begin{align*}
	\sigma^{n}\omega_{x}^{\A}&=(\ldots,\tau_{\vec{a}}^{n}\tau_{\vec{a}}^{-1}x^{\otimes d},\underset{\bullet}{\tau_{\vec{a}}^{n}x^{\otimes d}},\tau_{\vec{a}}^{n}\tau_{\vec{a}}x^{\otimes d},\ldots)\\
	&=\tilde{\tau}_{\vec{a}}(\ldots,\tau_{\vec{a}}^{n}x^{\otimes d},\underset{\bullet}{\tau_{\vec{a}}^{n}x^{\otimes d}},\tau_{\vec{a}}^{n}x^{\otimes d},\ldots),
\end{align*}
where $\tilde{\tau}_{\vec{a}}=\cdots\times\tau_{\vec{a}}^{-1}\times\underset{\bullet}{id_{X^d}}\times \tau_{\vec{a}}\times \tau_{\vec{a}}^{2}\times\cdots$. Thus, we have $$N_{\infty}(X,\A)=\{(\ldots,\tau_{\vec{a}}^{-1}\y,\underset{\bullet}{\y},\tau_{\vec{a}}\y,\ldots):\y\in N_{\A}(X,T)\}.$$
We define
$$\phi:(N_{\A}(X,T),\tau_{\vec{a}})\rightarrow(N_{\infty}(X,\A),\sigma),\y\mapsto(\ldots,\tau_{\vec{a}}^{-1}\y,\underset{\bullet}{\y},\tau_{\vec{a}}\y,\ldots),$$ and it is an isomorphism. Thus $(N_{\infty}(X,\A),\sigma)\cong(N_{\A}(X,T),\tau_{\vec{a}})$, and
$$(N_{\infty}(X,\A),\langle T^{\infty},\sigma\rangle)\cong (N_{\A}(X,T),\langle T^{(d)},\tau_{\vec{a}}\rangle).$$

\medskip

In this case, we have the following propery about recurrent points.

\begin{thm}\label{rec-N}
	$|Rec(X,T)|=1$ if and only if $|Rec(N_{\A}(X,T),\langle T^{(d)},\tau_{\vec{a}}\rangle)|=1$.
\end{thm}
\begin{proof}
	Necessity is obvious.
	
	Now assume that $Rec(X,T)=\{x\}$. If $(y_1,\ldots,y_d)\in Rec(N_{\A}(X,T),\langle T^{(d)},\tau_{\vec{a}}\rangle)$, then there exist $m_i,n_i$ such that $$\lim_{i\rightarrow\infty}(\tau_{\vec{a}})^{m_i}(T^{(d)})^{n_i}(y_1,\ldots,y_d)=(y_1,\ldots,y_d).$$
	i.e., $\lim_{i\rightarrow\infty}T^{m_{i}a_j+n_i}y_j=y_j$ for any $1\leq j\leq d$. Then $y_i\in Rec(X,T)=\{x\}$. 
	
	Thus, $(y_1,\ldots,y_d)=x^{(d)}$, we have $|Rec(N_{\A}(X,T),\langle T^{(d)},\tau_{\vec{a}}\rangle)|=1$.
	
\end{proof}

\medskip
But in gereral, Theorem \ref{rec-N} does not hold for the case of non-linear polynamials. 
There exists a system $(X,T)$ such that $|Rec(X,T)|=1$ but $|Rec(N_{\infty}(X,\A),\langle T^{\infty},\tilde{\sigma}\rangle)|\geq 2$ for some $\A$, see Example \ref{recexa}.

\medskip
\begin{exam}\label{recexa}
We construct the original system $(X,T)$ at first.

The first part of our system $(X,T)$ in $\R^2$ is a countable subset $A$ of the unit circle $\S^1$.
Let $a_i,i\in\Z$ and $a_{\infty}$ be different points in $\S^1$, with $\lim_{i\rightarrow\infty}a_{i}=\lim_{i\rightarrow\infty}a_{-i}=a_{\infty}$, see Figure 1 and put $A=\{a_i:i\in\Z\}\cup\{a_{\infty}\}$. Define the restriction of $T$ to the set $A$ by putting $T(a_i)=a_{i+1}$ and $T(a_{\infty})=a_{\infty}$.

\begin{center}
\begin{tikzpicture}
\coordinate (ain) at (2.5,0);
\coordinate (a0) at (2.5,5);
\coordinate (o) at (2.5,2.5);
\coordinate (d1) at ($ (o)!1!30:(a0) $);
\coordinate (d2) at ($ (o)!1!60:(a0) $);
\coordinate (d3) at ($ (o)!1!90:(a0) $);
\coordinate (d4) at ($ (o)!1!120:(a0) $);
\coordinate (d5) at ($ (o)!1!150:(a0) $);
\coordinate (a5) at ($ (o)!1!30:(ain) $);
\coordinate (a4) at ($ (o)!1!60:(ain) $);
\coordinate (a3) at ($ (o)!1!90:(ain) $);
\coordinate (a2) at ($ (o)!1!120:(ain) $);
\coordinate (a1) at ($ (o)!1!150:(ain) $);
\fill (a0) node [above left] {$a_0$} circle (1.5pt)
(d1) node [above left] {$a_{-1}$} circle (1.5pt)
(d2) node [above left] {$a_{-2}$} circle (1.5pt)
(d3) node [above left] {$a_{-3}$} circle (1.5pt)
(d4) node [below left] {$a_{-4}$} circle (1.5pt)
(d5) node [below left] {$a_{-5}$} circle (1.5pt)
(a5) node [below right] {$a_{5}$} circle (1.5pt)
(a4) node [below right] {$a_{4}$} circle (1.5pt)
(a3) node [above right] {$a_{3}$} circle (1.5pt)
(a2) node [above right] {$a_{2}$} circle (1.5pt)
(a1) node [above right] {$a_{1}$} circle (1.5pt)
(ain) node [below right] {$a_{\infty}$} circle (1.5pt);


\draw [shorten <=1mm][shorten >=1mm][-{Stealth[]}](a0) arc (90:60:2.5cm);
\draw [shorten <=1mm][shorten >=1mm][-{Stealth[]}](a1) arc (60:30:2.5cm);
\draw [shorten <=1mm][shorten >=1mm][-{Stealth[]}](a2) arc (30:0:2.5cm);
\draw [shorten <=1mm][shorten >=1mm][-{Stealth[]}](a3) arc (360:330:2.5cm);
\draw [shorten <=1mm][shorten >=1mm][-{Stealth[]}](a4) arc (330:300:2.5cm);
\draw [shorten <=1mm][shorten >=1mm][-{Stealth[]}](d5) arc (240:210:2.5cm);
\draw [shorten <=1mm][shorten >=1mm][-{Stealth[]}](d4) arc (210:180:2.5cm);
\draw [shorten <=1mm][shorten >=1mm][-{Stealth[]}](d3) arc (180:150:2.5cm);
\draw [shorten <=1mm][shorten >=1mm] [-{Stealth[]}](d2) arc (150:120:2.5cm);
\draw [shorten <=1mm][shorten >=1mm] [-{Stealth[]}](d1) arc (120:90:2.5cm);
\draw [
line width=1pt,
dash pattern=on 0.0001pt off 4pt,
line cap=round] (d5) arc (240:250:2.5cm);
\draw [
line width=1pt,
dash pattern=on 0.0001pt off 3pt,
line cap=round] (ain) arc (270:253:2.5cm);
\draw [
line width=1pt,
dash pattern=on 0.0001pt off 4pt,
line cap=round] (a5) arc (300:290:2.5cm);
\draw [
line width=1pt,
dash pattern=on 0.0001pt off 3pt,
line cap=round] (ain) arc (270:287:2.5cm);
\end{tikzpicture}
\end{center}
\begin{center}
	Figure 1. The system $(A,T|_{A})$
\end{center}

\medskip
The second part of $(X,T)$ is formed by just one trajectory $\{x_i\}_{i\in\Z_+}$ lying in $\R^2$, with $T(x_i)=x_{i+1}$, approaching the set $A$. 
The trajectory will be chosen such that $$d(x_n,A):=\min\{d(x_n,y):y\in A\}\searrow 0,n\rightarrow\infty$$ where $d$ is the Euclidean metric.

Fix a family of pairwise disjoint open sets $U(a_n)$ containing $a_n$.
Now we are going to describe the trajectory in more details. For any $n\in\Z_+$, we take $$x_{n^2}\in U(a_0),x_{n^2+1}\in U(a_1),\ldots,x_{n^2+n}\in U(a_{n}),$$ $$x_{n^2+n+1}\in U(a_{-n}),x_{n^2+n+2}\in U(a_{-n+1}),\ldots,x_{n^2+2n}\in U(a_{-1}),x_{(n+1)^2}\in U(a_{0}).$$
 We provide also pictures, see Figure 2 and Figure 3.

\begin{center}
	\begin{tikzpicture}
		\coordinate (ain) at (2.5,0);
		\coordinate (a0) at (2.5,5);
		\coordinate (o) at (2.5,2.5);
		\coordinate (d1) at ($ (o)!1!40:(a0) $);
		\coordinate (d2) at ($ (o)!1!80:(a0) $);
		\coordinate (a2) at ($ (o)!1!100:(ain) $);
		\coordinate (a1) at ($ (o)!1!140:(ain) $);
\coordinate (x0) at (2.5,4.6);
\coordinate (x4) at (2.5,5.3);

\coordinate (d3) at ($ (o)!1!115:(a0) $);
\coordinate (d4) at ($ (o)!1!130:(a0) $);
\coordinate (d5) at ($ (o)!1!150:(a0) $);
\coordinate (a5) at ($ (o)!1!30:(ain) $);
\coordinate (a4) at ($ (o)!1!50:(ain) $);
\coordinate (a3) at ($ (o)!1!65:(ain) $);
\coordinate (x5) at (4.3,4.7);
\coordinate (x6) at (5.2,2.9);
\coordinate (x7) at (-.2,3);
\coordinate (x8) at (0.7,4.7);
\fill (x8) node [above left] {$x_8$} circle (1.5pt) (x7) 
node [below right] {$x_7$} circle (1.5pt) (x6) 
node [below left] {$x_6$} circle (1.5pt) (x5) 
node [above right] {$x_5$} circle (1.5pt) (x0) 
node [left] {$x_0$} circle (1.5pt) (a0) node [right] {$x_1$} circle (1.5pt) (a1) node [below] {$x_2$} circle (1.5pt) (d1) node [below] {$x_3$} circle (1.5pt) (x4) node [above] {$x_4$} circle (1.5pt) 
(d5) node [below left] {$a_{-4}$} circle (1.5pt)
(a5) node [below right] {$a_{4}$} circle (1.5pt)
;
\draw [shorten <=0.5mm][shorten >=0.5mm][->](x0)--(a0);
\draw [shorten <=1mm][shorten >=1mm][->](a0)--(a1);
\draw [shorten <=1mm][shorten >=1mm][->](a1)--(d1);
\draw [shorten <=1mm][shorten >=1mm][->](d1)--(x4);
\draw [shorten <=1mm][shorten >=1mm][->](x4)--(x5);
\draw [shorten <=1mm][shorten >=1mm][->](x5)--(x6);
\draw [shorten <=1mm][shorten >=1mm][->](x6)--(x7);
\draw [shorten <=1mm][shorten >=1mm][->](x7)--(x8);

\draw (a3) circle [radius=0.6];
\node at (5.7,.8) {$U(a_{3})$};
\draw (d3) circle [radius=0.6];
\node at (-0.7,.8) {$U(a_{-3})$};
\draw (a0) circle [radius=0.6];
\node at (3.7,5.8) {$U(a_{0})$};
\draw (a1) circle [radius=0.6];
\node at (5.7,4.7) {$U(a_{1})$};
\draw (a2) circle [radius=0.6];
\node at (6.3,2.7) {$U(a_{2})$};
\draw (d1) circle [radius=0.6];
\node at (-0.7,4.8) {$U(a_{-1})$};
\draw (d2) circle [radius=0.6];
\node at (-1.4,2.7) {$U(a_{-2})$};

\fill (ain) node [below right] {$a_{\infty}$} circle (1.5pt);

\draw [
line width=1pt,
dash pattern=on 0.0001pt off 4pt,
line cap=round] (d5) arc (240:250:2.5cm);
\draw [
line width=1pt,
dash pattern=on 0.0001pt off 3pt,
line cap=round] (ain) arc (270:253:2.5cm);
\draw [
line width=1pt,
dash pattern=on 0.0001pt off 4pt,
line cap=round] (a5) arc (300:290:2.5cm);
\draw [
line width=1pt,
dash pattern=on 0.0001pt off 3pt,
line cap=round] (ain) arc (270:287:2.5cm);

	\end{tikzpicture}
\end{center}
\begin{center}
	Figure 2. $\{x_0,x_1,\ldots,x_8\}$
\end{center}
\medskip

\begin{center}
	\begin{tikzpicture}
		\coordinate (ain) at (2.5,0);
		\coordinate (a0) at (2.5,5);
		\coordinate (o) at (2.5,2.5);
		\coordinate (d1) at ($ (o)!1!30:(a0) $);
		\coordinate (d2) at ($ (o)!1!60:(a0) $);
		\coordinate (d3) at ($ (o)!1!90:(a0) $);
		\coordinate (d4) at ($ (o)!1!120:(a0) $);
		\coordinate (d5) at ($ (o)!1!130:(a0) $);
		\coordinate (a5) at ($ (o)!1!230:(a0) $);
		\coordinate (a4) at ($ (o)!1!240:(a0) $);
		\coordinate (a3) at ($ (o)!1!270:(a0) $);
		\coordinate (a2) at ($ (o)!1!300:(a0) $);
		\coordinate (a1) at ($ (o)!1!330:(a0) $);
		
		\coordinate (x-1) at ($ (o)!0.9!0:(d1) $);
		\coordinate (x-2) at ($ (o)!0.9!0:(d2) $);
		\coordinate (x-3) at ($ (o)!0.9!0:(d3) $);
		\coordinate (x-4) at ($ (o)!0.9!0:(d4) $);
		\coordinate (x-5) at ($ (o)!0.9!0:(d5) $);
		\coordinate (x5) at ($ (o)!0.9!0:(a5) $);
		\coordinate (x4) at ($ (o)!0.9!0:(a4) $);
		\coordinate (x3) at ($ (o)!0.9!0:(a3) $);
		\coordinate (x2) at ($ (o)!0.9!0:(a2) $);
		\coordinate (x1) at ($ (o)!0.9!0:(a1) $);
		\coordinate (x0) at ($ (o)!0.9!0:(a0) $);

		\draw (a0) circle [radius=0.5];
		\node at (3.2,5.8) {$U(a_{0})$};
		\draw (a1) circle [radius=0.5];
		\node at (5,4.9) {$U(a_{1})$};
		\draw (a2) circle [radius=0.5];
		\node at (6,3.7) {$U(a_{2})$};
		\draw (a3) circle [radius=0.5];
		\node at (6.3,2.5) {$U(a_{3})$};
		\node at (5.9,1) {$U(a_{n})$};
		\draw (d1) circle [radius=0.5];
		\node at (0.2,5.1) {$U(a_{-1})$};
		\draw (d2) circle [radius=0.5];
		\node at (-1,3.6) {$U(a_{-2})$};
		\draw (d3) circle [radius=0.5]; \node at (-1.1,2.2) {$U(a_{-3})$};
		\node at (-0.9,1) {$U(a_{-n})$};
		\draw (a5) circle [radius=0.5];
		\draw (d5) circle [radius=0.5];
		\fill (a0) node [above] {$x_{n^2}$} circle (1.5pt);
		\fill (ain) node [below] {$a_{\infty}$} circle (1.2pt);
		\fill (a5) circle (1pt);
		\fill (d5) circle (1pt);
		\fill (x5) node [above left] {$x_{n^2+n}$} circle (1.5pt) 
		(x3)  circle (1.5pt) 
		(x2)  circle (1.5pt) 
		(x1)  circle (1.5pt) 
		(x-5) node [above right] {$x_{n^2+n+1}$} circle (1.5pt) 
		(x-3) circle (1.5pt) 
		(x-2) circle (1.5pt) 
		(x-1) circle (1.5pt) 
		(x0) node [below] {$x_{(n+1)^2}$} circle (1.5pt);
				
		\draw [
		line width=1pt,
		dash pattern=on 0.0001pt off 3pt,
		line cap=round][shorten <=0.5mm][shorten >=0.5mm][-{Stealth[]}] (x3) arc (360:320:2.25cm);

		\draw [
line width=1pt,
dash pattern=on 0.0001pt off 3pt,
line cap=round][shorten <=0.5mm][shorten >=0.5mm][-{Stealth[]}] (x-5) arc (220:180:2.25cm);
		
		\draw [shorten <=0.5mm][shorten >=0.5mm][->](a0)--(x1);
		\draw [shorten <=1mm][shorten >=1mm][->](x1)--(x2);
		\draw [shorten <=1mm][shorten >=1mm][->](x2)--(x3);
		\draw [shorten <=1mm][shorten >=1mm][->](x5)--(x-5);
		\draw [shorten <=1mm][shorten >=1mm][->](x-3)--(x-2);
		\draw [shorten <=1mm][shorten >=1mm][->](x-2)--(x-1);
		\draw [shorten <=1mm][shorten >=1mm][->](x-1)--(x0);

		\draw [
		line width=1pt,
		dash pattern=on 0.0001pt off 3pt,
		line cap=round] (d5) arc (220:253:2.5cm);
		\draw [
		line width=1pt,
		dash pattern=on 0.0001pt off 2pt,
		line cap=round] (ain) arc (270:253:2.5cm);
		\draw [
		line width=1pt,
		dash pattern=on 0.0001pt off 3pt,
		line cap=round] (a5) arc (320:290:2.5cm);
		\draw [
		line width=1pt,
		dash pattern=on 0.0001pt off 2pt,
		line cap=round] (ain) arc (270:290:2.5cm);
		
	\end{tikzpicture}
\end{center}
\begin{center}
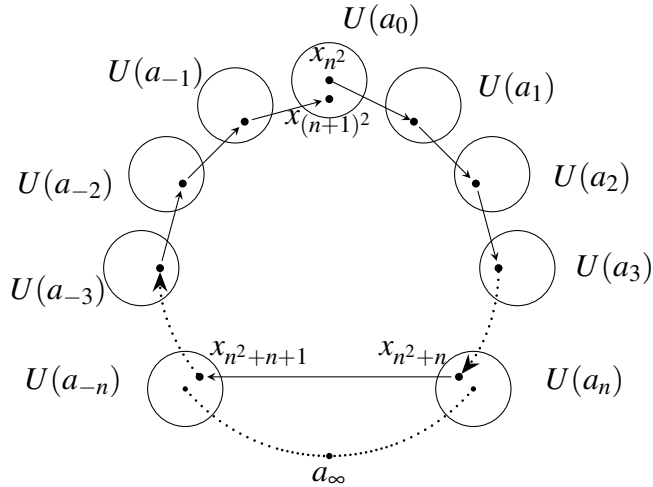
	Figure 3. $\{x_{n^2},\ldots,x_{n^2+n},x_{n^2+n+1},\ldots,x_{(n+1)^2}\}$
\end{center}
\medskip

The map $T:X\rightarrow X$ is continuous: when constructing the trajectory $\{x_n\}_{n\in\Z}$ approaching the set $A$, we use the general rule that the consecutive points $x_i,x_{i+1}$ are contained in $U(a_j)$ and $U(a_{j+1})$, respectively, for some $j\in\Z$. In rare exceptions from this rule the points $x_i,x_{i+1}$ are either `far away' from the set $A$ or they both are `close' to the fixed point $a_{\infty}$. Always, when $x_i,x_{i+1}$ is such exception, the point $x_i$ is on the `right side' of $a_{\infty}$ and $x_{i+1}$ is on the `left side' of $a_{\infty}$. In fact we will have infinitely many exceptions from what we call the general rule, but they will be closer and closer to the fixed point $a_{\infty}$. Therefore the continuity of $T$ will not be violated.

\medskip
It follows from the construction of $(X,T)$ that $Rec(X,T)=\{a_{\infty}\}$. 
Now we consider the induced system $(N_{\infty}(X,\A),\langle T^{\infty},\sigma\rangle)$ where $\A=\{n^2\}$. Note that $$\omega_{x_0}^{\A}=(T^{n^2}x_0)_{n\in\Z}=(x_{n^2})_{n\in\Z}.$$
On the one hand, 
we have $$\sigma^{n}\omega_{x_0}^{\A}\rightarrow a_0^{(\infty)}\in\ol{\O(\omega_{x_0}^{\A},\sigma)}\subseteq N_{\infty}(X,\A),\sigma a_0^{(\infty)}=a_0^{(\infty)}$$
where $a_0^{(\infty)}=(\ldots,a_0,a_0,\ldots)$.
So $a_0^{(\infty)}\in Rec(N_{\infty}(X,\A),\langle T^{\infty},\sigma\rangle)$. 
On the other hand, $a_{\infty}^{(\infty)}:=(\ldots,a_{\infty},a_{\infty},\ldots)$ is naturally a recurrent point under $\langle T^{\infty},\sigma\rangle$-action.

Thus, $a_0^{(\infty)},a_{\infty}^{(\infty)}\in Rec(N_{\infty}(X,\A),\langle T^{\infty},\sigma\rangle)$ and $|Rec(\langle T^{\infty},\sigma\rangle)|\geq 2$. 
\end{exam}

\bigskip

\section{The proof of Theorem B}
In this section, we prove Theorem B.
Recall that $F$ is thickly syndetic in $\Z$ iff for every $N\in\N$ the positions where interval with length $N$ runs begin form a syndetic set; 
$F\in\Z^2$ is thickly syndetic iff for every $M,N\in\N$ the positions where block with size $M\times N$ runs begin form a syndetic set of $\Z^2$. 

\medskip
Suppose that $S\subset\Z$ is thickly syndetic. 
Let $x:=1_{\Z\backslash S}\in\{0,1\}^{\Z}$. i.e.,
\begin{equation*}
	x_n=
	\begin{cases}
		0, & \text{if } n\in S;\\
		1, & \text{else.}
	\end{cases}
\end{equation*}
Consider the Bebutov system generated by $x$; that is, let $$X_S=\overline{\{T^{n}x:n\in\Z\}},$$ the orbit closure of $x$. 
Then $(X_S,T)$ is a transitive subshift where $T$ is the shift map. Let $x,y\in X_S$, the metric $\rho$ on $X$ is defined by \begin{equation*}
	\rho(x,y)=
	\begin{cases}
		\dfrac{1}{1+\min\{|n|:x_n\ne y_n\}}, & \text{if } x\ne y;\\
		0, & \text{if } x=y.
	\end{cases}
\end{equation*}

\begin{prop}\label{X_Smin}
	$Min(X,T)=\{{\bf 0}\}$. 
\end{prop}
\begin{proof}
	It is obvious that ${\bf 0}$ is a minimal point. For any $y\in X_S$, there exist $n_i\rightarrow\infty$ such that $T^{n_i}x\rightarrow y$ where $x=1_{\Z\setminus S}$. Clearly, the word $0^{(n)}$ of arbitrary length $n$ occurs in $y$ infinitely. Thus, ${\bf 0}\in\ol{\O(y,T)}$ which implies that ${\bf 0}$ is the unique minimal point.
\end{proof}

\medskip

Let $d\in\N$ and $\A=\{p_1,p_2,\ldots,p_d\}$ be a collection of integral polynomials with $p_i(0)=0,1\leq i\leq d$.
Then we have

\begin{prop}\label{min-N_infty}
	$Min(N_{\infty}(X_S,\A),\langle T^{\infty},\sigma\rangle)=\{{\bf 0}^{(\infty)}\}$. 
\end{prop}
\begin{proof}
	By Proposition \ref{X_Smin} and Theorem \ref{min-N}, 
    we have that ${\bf 0}^{(\infty)}$ is the only minimal point of $(N_{\infty}(X_S,\A),\langle T^{\infty},\sigma\rangle)$.
\end{proof}

\medskip

\begin{lem}\label{min-ts}
Let $(X,\langle T,S\rangle)$ be a $\Z^2$-action t.d.s. If $x$ is a transitive point, $K$ is the unique minimal subset. Then for any neighborhood $U$ of $K$, $$N_{\Z^2}(x,U):=\{(m,n):T^{m}S^{n}x\in U\}$$ is thickly syndetic.
\end{lem}
\begin{proof}
Let $U$ be any neighborhood of $K$. Since $K$ is the unique minimal subset of $X$, $U$ contains all closed invariant subsets of $X$. Then $X=\bigcup_{i,j=-\infty}^{\infty}T^{-i}S^{-j}U$. By the compactness of $X$, $X=\bigcup_{m,n=-N}^{N}T^{-i}S^{-j}U$ for some $N\in\N$. Hence, $$N_{\Z^2}(x,U)=\{(m,n):T^{m}S^{n}x\in U\}$$ is a syndetic set in $\Z^2$. 

\medskip
We may suppose $U_i\subseteq U$ is neighborhood of $K$ such that $T^{j_1+k_1}S^{j_2+k_2}x\in U$ for any $1\leq k_1,k_2\leq i$ whenever $T^{j_1}S^{j_2}x\in U_i$. Since for $i\in\N$, $N_{\Z^2}(x,U_i)$ is syndetic, we have $N_{\Z^2}(x,U)$ is thickly syndetic.

\end{proof}

\medskip

Now we are ready to prove Theorem B.

\begin{proof}[Proof of Theorem B]
	Let $d\in\N$, $p_i$ be integral polynamials with $p_i(0)=0,1\leq i\leq d$, and $S\subseteq \Z$ be thickly syndetic. Assume $(X_S,T)$ is the t.d.s. defined at the beginning of the section. Let $x=1_{\Z\setminus S}$. We know that $$N_{T}(x,B_1({\bf 0}))=\{n\in\Z:T^{n}x\in B_1({\bf 0})\}=S$$ where $B_1({\bf 0})=\{y\in X_S:\rho({\bf 0},y)<1\}:=\{y\in X_S:y_1=0\}$. 
	
	\medskip
	By Theorem \ref{min-N_infty}, the induced system $(N_{\infty}(X_S,\A),\langle T^{\infty},\sigma\rangle)$ has a unique minimal point ${\bf 0}^{(\infty)}$. Then for any neighborhood $\widetilde{U}$ of ${\bf 0}^{(\infty)}$, We have that 
	$$N^{\A}_{\Z^2}(x,U)=\{(m,n)\in\Z^2:(T^{\infty})^{m}\sigma^{n}\omega_x^{\A}\in \widetilde{U}\}$$
    is thickly syndetic in $\Z^2$ (by Lemma \ref{min-ts}) where $\widetilde{U}=\prod_{j=-\infty}^{-1}X^d\times U^{d}\times \prod_{j=1}^{\infty}X^d$, $U\subset X$ is a neighborhood containing ${\bf 0}$.  Specially, let
    $$\widetilde{U}=\prod_{j=-\infty}^{-1}X^d\times (B_1({\bf 0}))^{d}\times \prod_{j=1}^{\infty}X^d.$$
    It follows that
    \begin{align*}
    N^{\A}_{\Z^2}(x,B_1({\bf 0}))
    &= \{(m,n)\in\Z^2:(T^{\infty})^{m}\sigma^{n}\omega_x^{\A}\in \widetilde{U}\}\\
    &= \{(m,n)\in\Z^2:T^{m+p_1(n)}x\in B_1({\bf 0}),\ldots,T^{m+p_d(n)}x\in B_1({\bf 0})\}\\
    &= \{(m,n)\in\Z^2:m+p_1(n),\ldots,m+p_d(n)\in S\}.
    \end{align*}
Thus, $$\{(m,n)\in\Z^2:m+p_1(n),m+p_2(n),\ldots,m+p_d(n)\in S\}$$ is a thickly syndetic set of $\Z^2$.

\end{proof}

\bigskip

\section{The proof of Theorem C}

In this section, we construct a t.d.s. $(X,T)$ that satisfies the Theorem C. 
\subsection{The construction of the example}
Let $\Omega=\{0,1\}^{\Z^+}$. The mapping $T:\Omega\rightarrow\Omega$ by $(T x_i):=x_{i+1},i\in\Z^+$. $(\Omega,T)$ is a one-sided shift system.

Assume 
\begin{align*}
	A_1&=(1),\\
	A_2&=(10^{(b_1)}101),\\
	\ldots&\\
	A_{n+1}&=A_n0^{(b_n)}A_n0^{(n)}A_n,
\end{align*}
Let $b_n>15a_{n-1}$ where $a_n=|A_n|$ is the length of $A_n$.

Let  $x=\lim_{n\rightarrow\infty}A_n^{\infty}$ and $X=\ol{\O(x,T)}$, then $(X,T)$ is a subshift.
By the construction, $x$ is a recurrent point and $(X,T)$ is transitive.
Let $x,y\in X$, the metric $\rho$ on $X$ is defined by 
\begin{equation*}
	\rho(x,y)=
	\begin{cases}
		1/(1+\min\{n:x_n\ne y_n\}), & \text{if } x\ne y;\\
		0, & \text{if } x=y.
	\end{cases}
\end{equation*}


\begin{prop}
	$(X,T)$ is strong mixing.
\end{prop}

\begin{proof}
	By the construction of $x$, $A_{n}0^{(m)}A_n$ occurs in $x$ for any $m\geq n$, hence, $N_{T}([A_n],[A_n])$ is a cofinite set where $[A_n]:=\{x\in X:x[0;|A_n-1|]=A_n\}$. 
	
	Since $(X,T)$ is transitive, for any non-empty open set $U,V$, there exists $n\in\N$ such that $C=U\cap T^{-n}V\ne\varnothing$. Then
	$$N_{T}(U,V)=N_{T}(T^{-n}U,T^{-n}V)\supset N_{T}(T^{-n}C,C)\supset n+N(C,C).$$
	Assume that $T^{k}x\in C$ and take $j$ that $T^{k}([A_j])\subset C$, we have $$N_{T}(U,V)\supset n+N_{T}([A_k],[A_k]),$$ which implies that the system is strong mixing.
\end{proof}

\begin{prop}
	For any $y\in X, y\ne {\bf 0}$, $(y,y)$ is not $T\times T^2$-recurrent. 
\end{prop}
\begin{proof}	
	We assume $y\ne {\bf 0}$ is recurrent. Then the word $1$ occurs in $y$ infinitely.  Since the word $1$ always occurs in the form of $A_n$. It implies that  $A_n$ occurs in $y$ infinitely for a fixed $n$. Then there is a $k\in\Z^+$ such that $(T^{k}y)[0;|A_n|-1]=A_n$.
	
	We can take suitable $p_n$ such that
$$T^{k}y=C_1 0^{(b_{p_1})} C_2 0^{(b_{p_2})} C_3 0^{(b_{p_3})}\ldots C_m 0^{(b_{p_m})}\ldots$$
	where the word $0^{(b_{p_m})}$ does not occur in any $C_1,\ldots,C_m$. 
	By the structure of $X$, we have that $C_1=A_nD$ for some $D$, $C_m$ must be a subword of $A_{p_m}$ and $C_{m+1}$ must be a subword of $A_{p_m}0^{(m)}A_{p_m}$.
	We note $x(n)$ is the element at $n$-th position of $x$.
	
	\medskip
	\noindent{\bf Case 1.}
	If $\{n_i\}$ is a sequence such that $T^{n_i}y\rightarrow y$. 
	We know that $T^{n_i}(T^{k}y)\rightarrow T^{k}y$, $T^{k}y(0)=y(k)=1$. By passing to some subsequence if necessary, we may soppose the $n_i$-th position of $T^{k}y$ place at $C_{p_{m(i)}}$ for some $m(i)$. i.e.,  $$\sum_{k=1}^{m(i)-1}(|C_k|+b_{p_k})\leq n_i<\sum_{k=1}^{m(i)-1}(|C_k|+b_{p_k})+|C_{m(i)}|.$$ 
	
	Since $C_1$ is a subword of $A_{p_1}$ and $C_{m+1}$ is a subword of $A_{p_m}0^{(m)}A_{p_m}$. We have
	\begin{equation*}
		\begin{cases}
			|C_{m+1}|\leq|A_{p_m}0^{(m)}A_{p_m}|<b_{p_m}, \\
			\sum_{k=i}^{m-1}(|C_k|+b_{p_k})+|C_m|\leq b_{p_{m+1}}.
		\end{cases} 
	\end{equation*}
	
	Then 
	$$\sum_{k=1}^{m(i)-1}(|C_k|+b_{p_k})+|C_{m(i)}|< 2n_i<\sum_{k=1}^{m(i)}(|C_k|+b_{p_k}),$$
	and the $2n_i$-th position of $T^{k}y$ place at $0^{(b_{p_{m(i)}})}$.
	i.e., $$y(2n_{i}+k)=T^{k}y(2n_i)=0,\forall i\in\N.$$
	We have $\rho(T^{2n_i}y,y)>1/(k+1)$ for any $i\in\N$, which implies that $T^{2n_i}y\nrightarrow y$.
	
	\medskip
	\noindent{\bf Case 2.}
	If $\{n_i\}$ is a sequence such that $T^{2n_i}y\rightarrow y$. 
	We know that $T^{2n_i}(T^{k}y)\rightarrow T^{k}y$, $T^{k}y(0)=y(k)=1$. By passing to some subsequence if necessary, we may soppose the $2n_i$-th position of $T^{k}y$ place at $C_{p_{m(i)}}$ for some $m(i)\geq 3$. i.e.,  $$\sum_{k=1}^{m(i)-1}(|C_k|+b_{p_k})\leq 2n_i<\sum_{k=1}^{m(i)-1}(|C_k|+b_{p_k})+|C_{m(i)}|.$$ 
	
	Since $C_m$ is a subword of $A_{p_m}$ and $C_{m+1}$ is a subword of $A_{p_m}0^{(m)}A_{p_m}$, we have
	$$\sum_{k=1}^{m(i)-2}(|C_k|+b_{p_k})+|C_{m(i-1)}|< n_i<\sum_{k=1}^{m(i)-1}(|C_k|+b_{p_k}),$$
	the $n_i$-th position of $T^{k}y$ always place at $0^{(b_{p_{m(i)-1}})}$.
	Then  $y(n_{i}+k)=T^{k}y(n_i)=0,\forall i\in\N.$
	We have $\rho(T^{n_i}y,y)>1/(k+1)$ for any $i\in\N$, which implies that $T^{n_i}y\nrightarrow y$.
\end{proof}

By the propositions above, the set of multiple recurrent points of $(X,T)$ is $\{{\bf 0}\}$, so not dense in $X$. This finishes the proof of Theorem C.

\subsection{An question}
It remains to be a question that
\begin{ques}
	Let $(X,T)$ be transitive. Under which 
	condition of $(X,T)$, the set $$\{x\in X:\forall U\ni x,\exists\  n \ \text{s.t.}\ T^{p_1(n)}x\in U,\ldots,T^{p_d(n)}x\in U\}$$ will always be dense in $X$ for all integral polynomials $p_1,\ldots,p_d$ vanishing at zero?
\end{ques}
We denote the effective condition as $P$. 
It is clear that $P$ can be ``$(X,T)$ is minimal'' or ``$(X,T)$ is an $E$-system \footnote{$(X,T)$ is called an {\it E-system} if it is transitive and there is an $T$-invariant measure $\mu$ such that $supp(\mu):=\{x\in X:\text{for any neighborhood}\ U \ \text{of}\  x,\mu(U)>0\}=X$.}''.
And Theorem C has mean that $P$ can not be any mixing property of $(X,T)$.

\end{document}